\documentclass{amsproc}
\usepackage{breqn,float,amssymb,pdflscape,rotating,multirow,amsmath,empheq,xcolor}
\makeatletter
\@namedef{subjclassname@2020}{%
  \textup{2020} Mathematics Subject Classification}
\makeatother
\newcommand*\widefbox[1]{\fbox{\hspace{1em}#1\hspace{1em}}}
\definecolor{myblue}{rgb}{.8, .8, 1}

\newtheorem{theorem}{Theorem}[section]

\newtheorem{proposition}[theorem]{Proposition}

\theoremstyle{definition}

\newtheorem{example}[theorem]{Example}

\theoremstyle{remark}

\allowdisplaybreaks
\begin{document}

\title{Special Function Representation of Dickson Polynomials}


\author{Robert Reynolds}
\address[Robert Reynolds]{Department of Mathematics and Statistics, York University, Toronto, ON, Canada, M3J1P3}
\email[Corresponding author]{milver@my.yorku.ca}
\thanks{}


\subjclass[2020]{Primary  30E20, 33-01, 33-03, 33-04}

\keywords{Dickson polynomials, Cauchy integral, incomplete gamma function}

\date{}

\dedicatory{}

\begin{abstract}
Generating functions and functional equations of Dickson polynomials of the first and second kind are derived and continued analytically. These formulae are expressed in terms of the incomplete gamma function over complex variables of the parameters involved. Special cases are evaluated in terms of composite incomplete gamma functions and mathematical constants.
\end{abstract}

\maketitle
\section{Introduction}
Dickson polynomials \cite{qu} are important in both theory and applications. The study of these polynomials  led to a 70-year research breakthrough in combinatorics \cite{ding}, gave a family of perfect nonlinear functions for cryptography \cite{ding}, generated good linear codes \cite{carlet,yuan} for data communication and storage, and produced optimal signal sets for code division multiple access communications \cite{ding1}. The monograph \cite{lidl} and the aforementioned references contain additional details about Dickson polynomials. .\\\\
In the present paper one will find known and new mathematical formulae involving Dickson polynomials where the main tool used to derive the extended generating functions is contour integration \cite{reyn4}. This paper starts by itemizing preliminary equations present in current literature and used to derive the equations presented. Equations involving contour integral representations are also derived and used in aiding with derivations. Next we derive the main theorems involving Dickson polynomials in terms of the incomplete gamma function. This section is followed by the derivations of more generating functions in terms of the Gamma function. We end with some special cases of the generating functions in terms mathematical constants and concluding remarks.
\section{PRELIMINARY EQUATIONS}
The generalized Cauchy's integral formula \cite{reyn4} is given by

\begin{equation}\label{intro:cauchy}
\frac{y^k}{\Gamma(k+1)}=\frac{1}{2\pi i}\int_{C}\frac{e^{wy}}{w^{k+1}}dw.
\end{equation}
where $k,a,w,y \in \mathbb{C}$, $C$ is in general, an open contour in the complex plane where the bilinear concomitant has the same value at the end points of the contour. 
\subsection{THE INCOMPLETE GAMMA~FUNCTION}

The incomplete gamma functions~\cite{dlmf}, $\gamma(a,z)$ and $\Gamma(a,z)$, are defined by
\begin{equation}
\gamma(a,z)=\int_{0}^{z}t^{a-1}e^{-t}dt
\end{equation}
and
\begin{equation}
\Gamma(a,z)=\int_{z}^{\infty}t^{a-1}e^{-t}dt
\end{equation}
where $Re(a)>0$. The~incomplete gamma function has a recurrence relation given by
\begin{equation}
\gamma(a,z)+\Gamma(a,z)=\Gamma(a)
\end{equation}
where $a\neq 0,-1,-2,..$. The~incomplete gamma function is continued analytically by
\begin{equation}
\gamma(a,ze^{2m\pi i})=e^{2\pi mia}\gamma(a,z)
\end{equation}
and
\begin{equation}\label{eq:7}
\Gamma(a,ze^{2m\pi i})=e^{2\pi mia}\Gamma(a,z)+(1-e^{2\pi m i a})\Gamma(a)
\end{equation}
where $m\in\mathbb{Z}$, $\gamma^{*}(a,z)=\frac{z^{-a}}{\Gamma(a)}\gamma(a,z)$ is entire in $z$ and $a$. When $z\neq 0$, $\Gamma(a,z)$ is an entire function of $a$ and $\gamma(a,z)$ is meromorphic with simple poles at $a=-n$ for $n=0,1,2,...$ with residue $\frac{(-1)^n}{n!}$. These definitions are listed in Section~8.2(i) and (ii) in~\cite{dlmf}.
The incomplete gamma functions are particular cases of the more general hypergeometric and Meijer G functions see section (5.6) and equation (6.9.2) in \cite{erd}. 
Some Meijer G representations we will use in this work are given by;
\begin{equation}\label{g1}
\Gamma (a,z)=\Gamma (a)-G_{1,2}^{1,1}\left(z\left|
\begin{array}{c}
 1 \\
 a,0 \\
\end{array}
\right.\right)
\end{equation}
and
\begin{equation}\label{g2}
\Gamma (a,z)=G_{1,2}^{2,0}\left(z\left|
\begin{array}{c}
 1 \\
 0,a \\
\end{array}
\right.\right)
\end{equation}
from equations (2.4) and (2.6a) in \cite{milgram}.  We will also use the derivative notation given by;
\begin{equation}\label{g3}
\frac{\partial \Gamma (a,z)}{\partial a}=\Gamma (a,z) \log (z)+G_{2,3}^{3,0}\left(z\left|
\begin{array}{c}
 1,1 \\
 0,0,a \\
\end{array}
\right.\right)
\end{equation}
from equations (2.19a) in \cite{milgram}, (9.31.3) in \cite{grad} and equations (5.11.1), (6.2.11.1) and (6.2.11.2) in \cite{luke}, and (6.36) in \cite{aslam}.
\subsection{THE POCHHAMMER SYMBOL} 
These formulae are given in Section 5.2(iii) in \cite{dlmf} and the formula used in this work are given by;
%
\begin{equation}\label{eq:poch1}
(a)_n=\frac{\Gamma (a+n)}{\Gamma (a)}
\end{equation}
%
\begin{equation}\label{eq:poch2}
\frac{1}{(k)_{-n-p+1}}=(-1)^{n+p+1} (1-k)_{n+p-1}
\end{equation}
where $k,a\in\mathbb{C},n,p\in\mathbb{Z^+}$
%
%
%
\subsection{THE INCOMPLETE GAMMA FUNCTION CONTOUR INTEGRAL REPRESENTATION}
Using equation (\ref{intro:cauchy}) we replace $y$ by $y+\log(a)$ then multiply both sides by $e^{x y}$. Next we take the definite integral of both sides over $y\in[0,\infty)$ and simplify to get;
\begin{equation}\label{eq:igci}
\frac{a^{-x} (-x)^{-k-1} \Gamma (k+1,-x \log (a))}{\Gamma(k+1)}=-\frac{1}{2\pi i}\int_{C}\frac{a^w w^{-k-1}}{w+x}dw
\end{equation}
from equation (3.383.4) in \cite{grad} where $Re(w+x)<0,Re(x)>0$.
\subsection{GENERATING FUNCTIONS FOR DICKSON POLYNOMIALS}
In this section generating functions and functional equations of Dickson polynomials are listed based on formulae found in current literature.
\subsection{Dickson Polynomials Of The First Kind}
This is a form of equation (3.1) in \cite{zimmer}.
\begin{equation}\label{eq:dp1k}
\sum_{n=0}^{\infty}\sum_{j=0}^{\lfloor n/2 \rfloor}\frac{n (-a)^j z^n \binom{n-j}{j} x^{n-2 j}}{n-j}=\frac{z (x-2 a z)}{a z^2-x z+1}
\end{equation}
where $x,z,a\in\mathbb{C},|Re(x)|<1,|Re(z)|<1,|Re(a)|<1$.
\subsection{Dickson Polynomials Of The Second Kind}
This is given by equation (4.4) in \cite{qu}.
\begin{equation}\label{eq:dp2k}
\sum_{n=0}^{\infty}\sum_{j=0}^{\lfloor n/2 \rfloor}(-a)^j z^n \binom{n-j}{j} x^{n-2 j}=\frac{1}{a z^2-x z+1}
\end{equation}
where $x,z,a\in\mathbb{C},|Re(x)|<1,|Re(z)|<1,|Re(a)|<1$.
\subsection{Dickson Polynomials  Functional Equation Of The First Kind}
This equation is derived using Definition (9.6.1) and Theorem (9.6.3) in \cite{mullen}.
\begin{equation}\label{eq:fe1k}
\sum_{n=0}^{\infty}\sum_{j=0}^{\lfloor n/2 \rfloor}\frac{n (-b)^j z^n \binom{n-j}{j} \left(\frac{b}{u}+u\right)^{n-2 j}}{n-j}=\frac{2 b u z^2-b z+u^2 (-z)}{(u z-1) (u-b
   z)}
\end{equation}
where $z,b,u\in\mathbb{C},|Re(b)|<1,|Re(z)|<1$.
\subsection{Dickson Polynomials  Functional Equation Of The Second Kind}
This equation is derived using Definition (9.6.5) and Theorem (9.6.6) in \cite{mullen}.
\begin{equation}\label{eq:fe2k}
\sum_{n=0}^{\infty}\sum_{j=0}^{\lfloor n/2 \rfloor}(-\beta )^j z^n \binom{n-j}{j} \left(\frac{\beta }{y}+y\right)^{n-2 j}=-\frac{y}{(y z-1) (y-\beta  z)}
\end{equation}
where $y,z,\beta\in\mathbb{C},|Re(\beta)|<1,|Re(z)|<1$.
%
%
\section{DERIVATION OF EQUATIONS}
The procedure here is to derive equivalent contour integral expressions for the right-hand sides of equations in Section (3) through partial fractions and equation (\ref{eq:igci}). This procedure is repeated for the left-hand sides and then equated to yield the infinite sum of the Dickson polynomial in terms of the incomplete gamma function.\\
This method involves using a form of equation (\ref{intro:cauchy}) then multiply both sides by a function, then take a definite sum of both sides. This yields a definite sum in terms of a contour integral. A second contour integral is derived by multiplying equation (\ref{intro:cauchy}) by a function and performing some substitutions so that the contour integrals are the same.
\subsection{Derivation of equation (\ref{eq:eq1})}
Use equation (\ref{eq:dp1k}) and set $a\to \alpha, z \to w$ and apply equation (\ref{intro:cauchy}) to get
\begin{equation}
\sum_{n=0}^{\infty}\sum_{j=0}^{\lfloor n/2 \rfloor}\frac{n a^w (-\alpha )^j \binom{n-j}{j} x^{n-2 j} w^{-k+n-1}}{n-j}=\frac{1}{2\pi i}\int_{C}\frac{a^w w^{-k} (x-2 \alpha  w)}{\alpha  w^2-w
   x+1}dw
\end{equation}
Use equation (\ref{intro:cauchy}) to re-write the left-hand side. Next apply partial fractions to the right-hand and re-write using equation (\ref{eq:igci}) where the $x=\frac{x-\sqrt{x^2-4 \alpha }}{2 \alpha }$ and $x=\frac{\sqrt{x^2-4 \alpha }+x}{2 \alpha
   }$ and multiply by $\frac{x}{\sqrt{x^2-4 \alpha }}$. Then equate and simplify using equations (\ref{eq:poch1}) and (\ref{eq:poch2}).
\subsection{Derivation of equation (\ref{eq:eq2})}
Use equation (\ref{eq:dp1k}) and set $a\to \alpha, x \to w$ and apply equation (\ref{intro:cauchy}) to get
\begin{equation}
\sum_{n=0}^{\infty}\sum_{j=0}^{\lfloor n/2 \rfloor}\frac{n a^w (-\alpha )^j z^n \binom{n-j}{j} w^{-2 j-k+n-1}}{n-j}=\frac{1}{2\pi i}\int_{C}\frac{z a^w w^{-k-1} (2 \alpha  z-w)}{z (w-\alpha 
   z)-1}dw
\end{equation}
Use equation (\ref{intro:cauchy}) to re-write the left-hand side. Next apply partial fractions to the right-hand and re-write using equation (\ref{eq:igci}) where the $x=\frac{\alpha  z^2+1}{z}$. Then equate and simplify using equations (\ref{eq:poch1}) and (\ref{eq:poch2}).
\subsection{Derivation of equation (\ref{eq:eq3})}
Use equation (\ref{eq:dp1k}) and set $a \to w$ and apply equation (\ref{intro:cauchy}) to get
\begin{equation}
\sum_{n=0}^{\infty}\sum_{j=0}^{\lfloor n/2 \rfloor}\frac{(-1)^j n a^w z^n \binom{n-j}{j} w^{j-k-1} x^{n-2 j}}{n-j}=\frac{1}{2\pi i}\int_{C}\frac{z a^w w^{-k-1} (x-2 w z)}{w z^2-x z+1}dw
\end{equation}
Use equation (\ref{intro:cauchy}) to re-write the left-hand side. Next apply partial fractions to the right-hand and re-write using equation (\ref{eq:igci}) where the $x=\frac{x z-1}{z^2}$. Then equate and simplify using equations (\ref{eq:poch1}) and (\ref{eq:poch2}).
\subsection{Derivation of equation (\ref{eq:eq4})}
Use equation (\ref{eq:dp2k}) and set $a \to w$ and apply equation (\ref{intro:cauchy}) to get
\begin{equation}
\sum_{n=0}^{\infty}\sum_{j=0}^{\lfloor n/2 \rfloor}a^w (-\alpha )^j z^n \binom{n-j}{j} w^{-2 j-k+n-1}=\frac{1}{2\pi i}\int_{C}\frac{a^w w^{-k-1}}{-w z+\alpha  z^2+1}dw
\end{equation}
Use equation (\ref{intro:cauchy}) to re-write the left-hand side. Next apply partial fractions to the right-hand and re-write using equation (\ref{eq:igci}) where the $x=\frac{\alpha  z^2+1}{z}$. Then equate and simplify using equations (\ref{eq:poch1}) and (\ref{eq:poch2}).
\subsection{Derivation of equation (\ref{eq:eq5})}
Use equation (\ref{eq:dp2k}) and set $a\to \alpha, z \to w$ and apply equation (\ref{intro:cauchy}) to get
\begin{equation}
\sum_{n=0}^{\infty}\sum_{j=0}^{\lfloor n/2 \rfloor}a^w (-\alpha )^j \binom{n-j}{j} x^{n-2 j} w^{-k+n-1}=\frac{1}{2\pi i}\int_{C}\frac{a^w w^{-k-1}}{\alpha  w^2-w x+1}dw
\end{equation}
Use equation (\ref{intro:cauchy}) to re-write the left-hand side. Next apply partial fractions to the right-hand and re-write using equation (\ref{eq:igci}) where the $x=\frac{x-\sqrt{x^2-4 \alpha }}{2 \alpha }$ and $x=\frac{\sqrt{x^2-4 \alpha }+x}{2 \alpha
   }$ and multiply by $\frac{a^w w^{-k-1}}{\sqrt{x^2-4 \alpha }}$. Then equate and simplify using equations (\ref{eq:poch1}) and (\ref{eq:poch2}).
\subsection{Derivation of equation (\ref{eq:eq6})}
Use equation (\ref{eq:dp2k}) and set $a\to w$ and apply equation (\ref{intro:cauchy}) to get
\begin{equation}
\sum_{n=0}^{\infty}\sum_{j=0}^{\lfloor n/2 \rfloor}(-1)^j a^w z^n \binom{n-j}{j} w^{j-k-1} x^{n-2 j}=\frac{1}{2\pi i}\int_{C}\frac{a^w w^{-k-1}}{w z^2-x z+1}dw
\end{equation}
Use equation (\ref{intro:cauchy}) to re-write the left-hand side. Next apply partial fractions to the right-hand and re-write using equation (\ref{eq:igci}) where the $x=\frac{x z-1}{z^2}$ and multiply by $-\frac{1}{z^2}$. Then equate and simplify using equations (\ref{eq:poch1}) and (\ref{eq:poch2}).
\subsection{Derivation of equation (\ref{eq:eq7})}
Use equation (\ref{eq:fe1k}) and set $x\to u+b/u,a\to b,z\to w$ and apply equation (\ref{intro:cauchy}) to get
\begin{equation}
\sum_{n=0}^{\infty}\sum_{j=0}^{\lfloor n/2 \rfloor}\frac{n a^w (-b)^j \binom{n-j}{j} w^{-k+n-1} \left(\frac{b}{u}+u\right)^{n-2 j}}{n-j}=\frac{1}{2\pi i}\int_{C}\frac{a^w w^{-k-1} \left(2 u-w
   \left(b+u^2\right)\right)}{(u w-1) (b w-u)}dw
\end{equation}
Use equation (\ref{intro:cauchy}) to re-write the left-hand side. Next apply partial fractions to the right-hand and re-write using equations (\ref{eq:igci}) and (\ref{intro:cauchy}) where the $x=-u/b$ and $x=-1/u$ and multiply by $u/b$ for the first equation and $1/u$ for the second equation. Then equate and simplify using equations (\ref{eq:poch1}) and (\ref{eq:poch2}).
\subsection{Derivation of equation (\ref{eq:eq8})}
Use equation (\ref{eq:fe2k}) and set $z\to w$ and apply equation (\ref{intro:cauchy}) to get
\begin{equation}
\sum_{n=0}^{\infty}\sum_{j=0}^{\lfloor n/2 \rfloor}a^w (-\beta )^j \binom{n-j}{j} w^{-k+n-1} \left(\frac{\beta }{y}+y\right)^{n-2 j}=\frac{y a^w w^{-k-1}}{(w y-1) (\beta 
   w-y)}dw
\end{equation}
Use equation (\ref{intro:cauchy}) to re-write the left-hand side. Next apply partial fractions to the right-hand and re-write using equation (\ref{eq:igci}) where the $x=-1/y$ and $x=-y/\beta$ and multiply by $-\frac{y}{y^2-\beta }$ for the first equation and $\frac{y}{y^2-\beta }$ for the second equation. Then equate and simplify using equations (\ref{eq:poch1}) and (\ref{eq:poch2}).
\section{Generating Functions of Dickson polynomials and Functional Equations of First and Second Kind}
In this section extended generating functions of Dickson polynomials and functional equations of the first and second kind are listed.
\subsection{Generating Functions of Dickson polynomials  of First kind}
\begin{theorem}
For all $k,a\in\mathbb{C},|Re(\alpha)|<1,|Re(x)|<1,|m(\alpha)|<1,|Im(x)|<1$ then,
\begin{multline}\label{eq:eq1}
\sum\limits_{n=0}^{\infty}\sum\limits_{j=0}^{ \lfloor n/2 \rfloor}\frac{n }{n-j}\binom{n-j}{j} (-\alpha )^j x^{n-2 j} (1-k)_{n-1} \left(-\frac{1}{a}\right)^n\\
=\left(\frac{2}{a}\right)^k e^{\frac{a \left(x-\sqrt{x^2-4 \alpha }\right)}{2 \alpha }} 
    \left(-\frac{\Gamma \left(k,\frac{a \left(x-\sqrt{x^2-4 \alpha }\right)}{2 \alpha }\right)}{\left(\frac{x-\sqrt{x^2-4 \alpha }}{\alpha
   }\right)^k}-\frac{\exp \left(\frac{a \sqrt{x^2-4 \alpha }}{\alpha }\right) \Gamma \left(k,\frac{a \left(x+\sqrt{x^2-4 \alpha }\right)}{2 \alpha
   }\right)}{\left(\frac{x+\sqrt{x^2-4 \alpha }}{\alpha }\right)^k}\right)
    \end{multline}
       \end{theorem}
\begin{theorem}
For all $k,a\in\mathbb{C},|Re(\alpha)|<1,|Re(z)|<1,|m(\alpha)|<1,|Im(z)|<1$ then,
\begin{multline}\label{eq:eq2}
\sum\limits_{n=0}^{\infty}\sum\limits_{j=0}^{ \lfloor n/2 \rfloor}\frac{n}{n-j} \binom{n-j}{j}  \left(-a^2 \alpha \right)^j(1-k)_{n-1-2 j} \left(-\frac{z}{a}\right)^n\\
=\frac{e^{a
   \left(\frac{1}{z}+z \alpha \right)} k \left(z^2 \alpha -1\right) \Gamma \left(k,\left(\frac{1}{z}+z \alpha \right)
   a\right)+2 z^2 \alpha  \left(\left(\frac{1}{z}+z \alpha \right) a\right)^k}{a^k \left(\frac{1}{z}+z \alpha \right)^k
   \left(1+z^2 \alpha \right) k}
    \end{multline}
 \end{theorem}
\begin{theorem}
For all $k,a\in\mathbb{C},|Re(x)|<1,|Re(z)|<1,|m(x)|<1,|Im(z)|<1$ then,
\begin{multline}\label{eq:eq3}
\sum\limits_{n=0}^{\infty}\sum\limits_{j=0}^{ \lfloor n/2 \rfloor}\frac{n }{n-j}\binom{n-j}{j} a^{-j} x^{n-2 j} (1-k)_{j-1} z^n\\
=\frac{e^{\frac{a (-1+x z)}{z^2}} \left(x z \Gamma
   \left(1+k,\frac{a (-1+x z)}{z^2}\right)-2 k (-1+x z) \Gamma \left(k,\frac{a (-1+x z)}{z^2}\right)\right)}{a^k
   \left(\frac{-1+x z}{z^2}\right)^{k+1} k z^2}
    \end{multline}
    \end{theorem}
\subsection{Generating Functions of Dickson polynomials  of second kind}
\begin{theorem}
For all $k,a\in\mathbb{C},|Re(\alpha)|<1,|Re(z)|<1,|m(\alpha)|<1,|Im(z)|<1$ then,
\begin{multline}\label{eq:eq4}
\sum\limits_{n=0}^{\infty}\sum\limits_{j=0}^{ \lfloor n/2 \rfloor}\left(-\frac{\alpha }{x^2}\right)^j \binom{n-j}{j} (1-k)_{n-1}
   \left(-\frac{x}{a}\right)^n\\
   =-\frac{\left(\frac{2}{a}\right)^k 2 e^{\frac{a \left(x-\sqrt{x^2-4 \alpha }\right)}{2
   \alpha }} }{\sqrt{x^2-4 \alpha } k}\left(\frac{\Gamma \left(k+1,\frac{\left(x-\sqrt{x^2-4 \alpha }\right) a}{2 \alpha
   }\right)}{\left(\frac{x-\sqrt{x^2-4 \alpha }}{\alpha }\right)^{k+1}}-\frac{e^{\frac{a \sqrt{x^2-4 \alpha }}{\alpha }} \Gamma
   \left(k+1,\frac{\left(x+\sqrt{x^2-4 \alpha }\right) a}{2 \alpha }\right)}{\left(\frac{x+\sqrt{x^2-4 \alpha }}{\alpha
   }\right)^{k+1}}\right)
    \end{multline}
    \end{theorem}
\begin{theorem}
For all $k,a\in\mathbb{C},|Re(\alpha)|<1,|Re(x)|<1,|m(\alpha)|<1,|Im(x)|<1$ then,
\begin{multline}\label{eq:eq5}
\sum\limits_{n=0}^{\infty}\sum\limits_{j=0}^{ \lfloor n/2 \rfloor}\left(-\alpha  a^2\right)^j \binom{n-j}{j} (1-k)_{n-1-2 j} \left(-\frac{z}{a}\right)^n=-\frac{e^{a \left(\frac{1}{z}+z
   \alpha \right)} \Gamma \left(1+k,\left(\frac{1}{z}+z \alpha \right) a\right)}{a^k \left(\frac{1}{z}+z \alpha \right)^k
   \left(1+z^2 \alpha \right) k}
    \end{multline}
    \end{theorem}
\begin{theorem}
For all $k,a\in\mathbb{C},|Re(x)|<1,|Re(z)|<1,|m(x)|<1,|Im(z)|<1$ then,
\begin{multline}\label{eq:eq6}
\sum\limits_{n=0}^{\infty}\sum\limits_{j=0}^{ \lfloor n/2 \rfloor}\left(\frac{1}{a x^2}\right)^j \binom{n-j}{j} (1-k)_{j-1} (x z)^n=\frac{e^{\frac{a (-1+x z)}{z^2}} \Gamma
   \left(1+k,\frac{(-1+x z) a}{z^2}\right)}{\left(\frac{-1+x z}{z^2}\right)^{k+1} a^k z^2 k}
    \end{multline}
    \end{theorem}
\subsection{Generating Function of Dickson polynomials Functional equation of First kind}
\begin{theorem}
For all $k,a\in\mathbb{C},|Re(b)|<1,|Re(u)|<1,|Im(b)|<1,|Im(u)|<1$ then,
\begin{multline}\label{eq:eq7}
\sum\limits_{n=0}^{\infty}\sum\limits_{j=0}^{ \lfloor n/2 \rfloor}\frac{n }{n-j}\binom{n-j}{j} (-b)^j \left(\frac{b}{u}+u\right)^{2-j} (1-k)_{n-1}
   \left(-\frac{\frac{b}{u}+u}{a}\right)^n\\
   =\frac{\left(\frac{b}{a}\right)^k \left(2 \left(\frac{a}{b}\right)^k-e^{a/u}
   \left(\frac{u}{b}\right)^k \Gamma \left(1+k,\frac{a}{u}\right)-e^{\frac{a u}{b}} u^{-k} \Gamma \left(1+k,\frac{a
   u}{b}\right)\right)}{k}
    \end{multline}
    \end{theorem}
\subsection{Generating Function of Dickson polynomials Functional equation of second kind}
\begin{theorem}
For all $k,a\in\mathbb{C},|Re(\beta)|<1,|Re(y)|<1,|Im(\beta)|<1,|Im(y)|<1$ then,
\begin{multline}\label{eq:eq8}
\sum\limits_{n=0}^{\infty}\sum\limits_{j=0}^{ \lfloor n/2 \rfloor}\binom{n-j}{j}(-\beta )^j \left(y+\frac{\beta }{y}\right)^{n-2 j}  (1-k)_{n-1} \left(-\frac{1}{a}\right)^n\\
=\frac{e^{a/y}
   y^{k+2} \Gamma \left(1+k,\frac{a}{y}\right)-e^{\frac{a y}{\beta }} \left(\frac{\beta }{y}\right)^k \beta  \Gamma
   \left(1+k,\frac{y a}{\beta }\right)}{a^k k \left(y^2-\beta \right)}
    \end{multline}
    \end{theorem}
\section{Derivations In Terms Of The Gamma Function}
In this section we used equations (\ref{eq:eq1}) to (\ref{eq:eq8}) simultaneously to derive double summation formulae for the gamma function.
\begin{proposition}
\begin{multline}\label{eq:g1}
\sum_{n=0}^{\infty}\sum_{j=0}^{\lfloor n/2 \rfloor}\left(-\frac{1}{4}\right)^j \left(-\frac{x}{a}\right)^n \left(a
   (k-n) \left(-\frac{x}{a}\right)^2+\frac{n (2 a-k x)}{n-j}\right)
   \binom{n-j}{j} \Gamma (n-k)\\=-2 x \Gamma (1-k)
\end{multline}
\end{proposition}
\begin{proof}
Use equations (\ref{eq:eq1}) and (\ref{eq:eq4}) and solve simultaneously using $\alpha=x^2/4$.
\end{proof}
%
%
%
\begin{proposition}
\begin{multline}\label{eq:g2}
\sum_{n=0}^{\infty}\sum_{j=0}^{\lfloor n/2 \rfloor}\frac{1}{n-j}\left(-\frac{z}{a}\right)^n \left(-a^2 \alpha \right)^j \left((j-n) (2 j+k-n) z \left(1-z^2
   \alpha \right)\right. \\ \left.
   +a \left(n+n z^2 \alpha \right)\right) \binom{n-j}{j} \Gamma (n-2 j-k)
   =-2 a z^2 \alpha 
   \Gamma (-k)
\end{multline}
\end{proposition}
\begin{proof}
Use equations (\ref{eq:eq2}) and (\ref{eq:eq5}) and solve simultaneously.
\end{proof}
\begin{proposition}
\begin{multline}\label{eq:g3}
\sum_{n=0}^{\infty}\sum_{j=0}^{\lfloor n/2 \rfloor}\frac{ n }{n-j} \binom{n-j}{j}\left(\frac{1}{a x^2}\right)^j(x z)^n  (1-k)_{j-1}\\
=\frac{1}{a \Gamma (1-k)}\sum_{n=0}^{\infty}\sum_{j=0}^{\lfloor n/2 \rfloor}\binom{n-j}{j}
   \left(\frac{1}{a x^2}\right)^j z(x z)^n (a x+2 (j-k) z)  \Gamma (j-k)
\end{multline}
\end{proposition}
\begin{proof}
Use equations (\ref{eq:eq3}) and (\ref{eq:eq6}) and solve simultaneously.
\end{proof}
\subsection{Gamma Function In Terms Of Quotient Infinite Sums}
Here we use equation (\ref{eq:g3}) and solve for the gamma function and replace $k \to 1-k$ and simplify.
\begin{empheq}[box=\widefbox]{align*}
\Gamma (k)=\left(\frac{z}{a}\right)\frac{\sum\limits_{n=0}^{\infty}\sum\limits_{j=0}^{\lfloor n/2 \rfloor}\left(\frac{1}{a x^2}\right)^j (x z)^n (a x+2 (-1+j+k) z) \binom{n-j}{j} \Gamma (-1+j+k)}{\sum\limits_{n=0}^{\infty}\sum\limits_{j=0}^{\lfloor n/2 \rfloor}\frac{n }{n-j}\binom{n-j}{j}\left(\frac{1}{a x^2}\right)^j (x z)^n  (k)_{j-1}}
\end{empheq}
\begin{example}
Using the above equation when $k=1/2,a=2,x=1/2,z=1/3$ we get;
\begin{multline}
\Gamma \left(\frac{1}{2}\right)=\sqrt{\pi}=\left(\frac{1}{6}\right)\frac{\sum\limits_{n=0}^{\infty}\sum\limits_{j=0}^{\lfloor n/2 \rfloor} \binom{n-j}{j}\Gamma \left(j-\frac{1}{2}\right)\left(\frac{2}{3} \left(j-\frac{1}{2}\right)+1\right) 3^{-n} 2^{j-n} 
  }{\sum\limits_{n=0}^{\infty}\sum\limits_{j=0}^{\lfloor n/2 \rfloor}\frac{n  }{n-j}\binom{n-j}{j}\left(\frac{1}{2}\right)_{j-1} 3^{-n}2^{j-n} }
\end{multline}
\end{example}
\begin{example}
A double product of the exponential function.
\begin{multline}
\prod_{n=0}^{\infty}\prod_{j=0}^{\lfloor n/2 \rfloor}\exp \left(\frac{2^{-n-1} \left(2^{n+1}-1\right) n z \left(-e^{2 a}
   \alpha \right)^j \left(-e^{-a} z\right)^n \binom{n-j}{j} (0)_{-2
   j+n-1}}{(n+1) (n-j)}\right)\\
=2^{\frac{e^{-a}}{\alpha }} \left(\frac{\alpha 
   z^2+i \sqrt{\alpha } z+2}{\alpha  z^2-i \sqrt{\alpha }
   z+2}\right)^{\frac{i}{\sqrt{\alpha }}} e^{z-\frac{3 e^{-a} z^2}{\alpha ^2
   z^4+5 \alpha  z^2+4}} \left(1-\frac{3}{\alpha 
   z^2+4}\right)^{\frac{e^{-a}}{2 \alpha }}
\end{multline}
\end{example}
\begin{proof}
Use equation (\ref{eq:eq1}) and set $k=1,a=e^a$ then take the definite integral $z\in[z/2,z]$ and simplify. Next take the exponential function of both sides and simplify.
\end{proof}

%
\section{DISCUSSION}
In this paper, we have presented a method for deriving generating function involving Dickson polynomials along with some interesting functional equations using contour integration. The results presented were numerically verified for both real and imaginary and complex values of the parameters in the integrals using Mathematica by Wolfram.
\end{document}